\documentclass[12pt,a4paper]{amsart} 

\usepackage{latexsym}
\usepackage{amsmath}
\usepackage{amsthm}
\usepackage{latexsym}
\usepackage{amssymb}
\usepackage[a4paper , lmargin = {2cm} , rmargin = {2cm} , tmargin = {2.5cm} , bmargin = {2.5cm} ]{geometry}

\usepackage{hyperref}

\DeclareMathOperator{\lk}{lk}
\DeclareMathOperator{\Nor}{Nor}
\DeclareMathOperator{\Pc}{Pc}
\DeclareMathOperator{\GL}{GL}
\DeclareMathOperator{\Out}{Out}
\DeclareMathOperator{\Mod}{Mod}
\DeclareMathOperator{\Isom}{Isom}
\DeclareMathOperator{\card}{card}

\renewcommand{\phi}{\varphi}
\renewcommand{\emptyset}{\varnothing}

\newcommand{\F}{{\rm F}}             
\newcommand{\CAT}{{\rm CAT$(0)$}}

\newcommand{\G}{\mathcal G}

\makeatletter
\def\theenumi{\@roman\c@enumi}
\makeatother

\theoremstyle{plain}
\newtheorem*{NewPropositionA}{Proposition A}	
\newtheorem*{NewTheoremB}{Theorem B}
\newtheorem*{NewPropositionC}{Proposition C}	
\newtheorem*{NewTheoremD}{Theorem D}
\newtheorem*{NewCorollaryE}{Corollary E}

\newtheorem{lemma}[subsection]{Lemma}

\newtheorem{proposition}[subsection]{Proposition}
\theoremstyle{definition}
\newtheorem{definition}[subsection]{Definition}

\title[]{Abstract homomorphisms from locally compact groups to discrete groups}
\author{Linus Kramer and Olga Varghese}

\address{Linus Kramer\\
Department of Mathematics\\
M\"unster University\\ 
Einsteinstra\ss e 62\\
48149 M\"unster (Germany)}
\email{linus.kramer@uni-muenster.de}

\address{Olga Varghese\\
Department of Mathematics\\
M\"unster University\\ 
Einsteinstra\ss e 62\\
48149 M\"unster (Germany)}
\email{olga.varghese@uni-muenster.de}

\begin{document}
\begin{abstract}
We show that every abstract homomorphism $\varphi$ from a locally compact group $L$ to a graph product $G_\Gamma$,
endowed with the discrete topology, is either continuous or $\varphi(L)$ lies in a 'small' parabolic subgroup. 
In particular, every locally compact group topology on a graph product whose graph is not 'small' is discrete.
This extends earlier work by Morris-Nickolas.

We also show the following. If $L$ is a locally compact group and if $G$ is a discrete group which contains no infinite 
torsion group
and no infinitely generated abelian group, then every  abstract homomorphism $\varphi:L\to G$ is 
either continuous, or $\varphi(L)$ is contained in the normalizer of a finite nontrivial subgroup of $G$.
As an application we obtain results concerning the continuity of homomorphisms from locally compact groups to Artin and Coxeter groups.
\end{abstract}

\thanks{{Funded by the Deutsche 
Forschungsgemeinschaft (DFG, German Research Foundation) under Germany's 
Excellence Strategy EXC 2044--390685587, Mathematics M\"unster: Dynamics-Geometry-Structure.}}

\maketitle

\section{Introduction}
We investigate the following type of question.
\begin{quote}
{\em Let $L$ be a  locally compact group and let $G$ be a discrete group.
Under what conditions on the group $G$ is an abstract (i.e. not necessarily continuous) homomorphism 
$\varphi:L\to G$ automatically continuous?}
\end{quote}
There are many results in this direction in the literature, see \cite{Kramer}, \cite{Conner}, \cite{Dudley}, \cite{Morris} or \cite{Paolini}. 
In particular, Dudley \cite{Dudley} proved that every abstract homomorphism from a locally compact group to a free group
is automatically continuous.
This was generalized by Morris and Nickolas \cite{Morris}. They proved that every abstract homomorphism from a locally compact group to 
a free product of groups is either continuous, or the image of the homomorphism is conjugate to a subgroup of one of the 
factors of the free product.

Our first aim is to prove similar results for the case where the codomain $G$ of an abstract homomorphism 
$L\to G$ is a graph product of arbitrary groups.  
Given a simplicial graph $\Gamma=(V, E)$ and a collection of groups $\mathcal{G} = \{ G_u \mid u \in V\}$, the \emph{graph product} $G_\Gamma$ is defined as the quotient
\[ \left( \underset{u \in V}{\ast} G_u \right) / \langle \langle [G_v,G_w]\text{ for }\{v,w\}\in E \rangle \rangle.
\]
We call $G_\Gamma$ \emph{finite dimensional} if there exists a uniform bound on the sizes of cliques in $\Gamma$.

Throughout, $L$ denotes a Hausdorff locally compact group with identity component $L^\circ$, and $G$ denotes a discrete group.
We call $L$ \emph{almost connected} if the totally disconnected group $L/L^\circ$ is compact.
By an \emph{abstract homomorphism} we mean a group homomorphism between topological groups which is not assumed to be continuous.
We remark that every abstract homomorphism whose codomain is discrete is open.

\begin{NewPropositionA}
Let $\varphi:L\to G_\Gamma$ be an abstract homomorphism from an almost connected locally compact group $L$ to a finite dimensional graph product $G_\Gamma$. Then $\varphi(L)$ lies in a complete parabolic subgroup of $G_\Gamma$.
\end{NewPropositionA}

Using Proposition A, we show the following more general result.
\begin{NewTheoremB}
Let $\varphi$ be an abstract homomorphism from a locally compact group $L$ to a finite dimensional graph product $G_\Gamma$. 
Then either $\varphi$ is continuous, or $\varphi(L)$ lies in a 
conjugate of a parabolic subgroup $G_{S\cup\lk(S)}$, where $S\neq\emptyset$ is a clique.
If every composite $L\xrightarrow{\varphi}G_\Gamma\xrightarrow{r_v} G_v$ is continuous, then $\varphi$ is 
continuous.
\end{NewTheoremB}
In particular, every locally compact group topology on a finite dimensional graph product $G_\Gamma$ is discrete, unless
$\Gamma$ is contained in the link of a clique.
In the latter case, $G_\Gamma$ is a direct product of vertex groups and a smaller graph product, and then
a locally compact topology on $G_\Gamma$ may indeed be nondiscrete.

\smallskip
Our remaining results deal with a certain class $\G$ of discrete groups.
Let $\G$ denote the class of all groups $G$ with the following two properties:
\begin{enumerate}
 \item Every torsion subgroup $T\subseteq G$ is finite, and
 \item Every abelian subgroup $A\subseteq G$ is a (possibly infinite) direct sum of cyclic groups.
\end{enumerate}
The abelian subgroups $A$ in such a group are thus of the form $A=F\times\mathbb Z^{(J)}$,
where $F$ is a finite abelian group and $\mathbb Z^{(J)}$ is free abelian of (possibly infinite) rank $\card(J)$.
We remark that subgroups of free abelian groups are again free abelian \cite[A1.9]{HoMo}.

We study abstract homomorphisms from locally compact groups to groups in this class.
We show in Section 7 that the class $\G$ is huge. 
It is closed under finite products, under coproducts, and under passage to subgroups, see Proposition~\ref{LargeClass}.
For example, every finitely generated hyperbolic group, 
every right-angled Artin group, every
Artin group of finite type and every Coxeter group is in this class, see Propositions \ref{Hyperbolic}, \ref{Artin} and \ref{Coxeter}. 
Furthermore, the groups $\GL_n(\mathbb{Z})$, $\Out(F_n)$ and the mapping class groups $\Mod(S_g)$ of compact orientable surfaces 
of genus $g$ are in this class, see Proposition~\ref{OutFn}. Further examples of groups which are in the class $\G$ are diagram groups, see \cite[Theorem 16]{Guba}. In particular, the Thompson's group $F$ is a diagram group and this group contains a free abelian group which is not finitely generated.

We obtain the following results.
\begin{NewPropositionC}
Let $\varphi$ be an abstract homomorphism from a locally compact group $L$ to a group $G$ in the class $\mathcal{G}$. Then $\varphi$ factors through the canonical projection $\pi:L\to L/L^\circ$. If $L$ is almost connected, 
then $\varphi(L)$ is finite.  
\end{NewPropositionC}

\begin{NewTheoremD}
Let $\varphi$ be an abstract homomorphism from a locally compact group $L$ to a group $G$ in the class $\mathcal{G}$. 
Then either $\varphi$ is continuous, or $\varphi(L)$ lies in the normalizer of a finite non-trivial subgroup of $G$.
\end{NewTheoremD}
The following is an immediate consequence of Theorem D.
\begin{NewCorollaryE}
Every abstract homomorphism from a locally compact group $L$ to a torsion free group $G$ in the class $\mathcal{G}$ is continuous. 
In particular, every abstract homomorphism from a locally compact group to a right-angled Artin group or to an Artin group of finite type is continuous.
\end{NewCorollaryE}

Related results on abstract homomorphisms into right-angled Artin groups and into Artin groups of non-exceptional finite type were 
recently proved in \cite{Corson}.

Our proofs depend heavily on a theorem of Iwasawa on the structure of connected locally compact groups and on a theorem of van Dantzig on the existence of compact open subgroups in totally disconnected groups. For the proof of Proposition A we use the structure of the 
right-angled building $X_\Gamma$ associated to a graph product $G_\Gamma$.

\section*{Acknowledgment}
The authors thank the referee for careful reading of the manuscript and many helpful remarks, in particular concerning the infinite dimensional case in Theorem B.

\section{Graph products}\label{GraphProductsSection}
In this section we briefly present the main definitions and properties
concerning  graph products. These groups are defined by presentations of a special form. 

A \emph{simplicial graph} $\Gamma=(V,E)$ consists of a set $V$ of \emph{vertices}
and a set $E$ of $2$-element subsets
of $V$ which are called \emph{edges}. We allow infinite graphs.
Given a subset $S\subseteq V$, the \emph{graph generated by $S$} is the graph with vertex set $S$
and edge set $E|_S=\{\{v,w\}\in E\mid v,w\in S\}$.
We call $S$ a \emph{clique} if $E|_S=\{\{v,w\}\mid v,w\in S\text{ with }v\neq w\}=\binom{S}{2}$.
We count the empty set as a clique.
We say that $\Gamma$ has \emph{finite dimension} if there is a uniform upper bound on the cardinality of
cliques in $\Gamma$.
For a subset $S\subseteq V$ we define its \emph{link} as
\[
\lk(S)=\{w\in V\mid \{v,w\}\in E\text{ for all }v\in S\}.
 \]

\begin{definition}
Let $\Gamma$ be a simplicial graph, as defined above. Suppose that for every vertex $v\in V$ we are given a 
nontrivial\footnote{We need nontrivial vertex groups in order to obtain a building. Alternatively, one may remove all
vertices $v$ from $\Gamma$ whose vertex group $G_v$ is trivial, without changing the resulting graph product.}
abstract group $G_v$.
The \emph{graph product} $G_\Gamma$ is the group obtained from the free product of
the $G_v$, for $v\in V$,  by adding the commutator relations $gh=hg$ for all $g\in G_v$, $h\in G_w$ with $\left\{v,w\right\}\in E$,
i.e.

\[ \left( \underset{u \in V}{\ast} G_u \right) / \langle \langle [G_v,G_w]\text{ for }\{v,w\}\in E \rangle \rangle.
\]
\end{definition}
Graph products are special instances of graphs of groups, and in particular colimits in the category of groups
\cite[\S5]{Davis}.
We call the graph product \emph{finite dimensional} if $\Gamma$ has finite dimension as defined above, i.e. if
there is an upper bound on the size of cliques in $\Gamma$.

The first examples to consider are the extremes. If $E=\emptyset$, then $G_{\Gamma}$ is the free product of the 
groups $G_v$, for $v\in V$. On the other hand, if $E=\binom{V}{2}$ is the set of all $2$-element subsets of $V$,
then $G_\Gamma$ is the direct sum of the $G_v$, for $v\in V$. 
So graph products interpolate between free products and direct sums of groups.

\subsection{Parabolic subgroups}\label{ParabolicSubgroups}
Let $\Gamma=(V,E)$ be a simplicial graph, let $G_\Gamma$ denote the graph product of a family of groups $\{G_v\mid v\in V\}$
and let $S$ be a subset of $V$. The subgroup $G_S$ of $G_\Gamma$ generated by the $G_v$, for $v\in S$,
is again a graph product,  corresponding to the subgraph $\Gamma'=(S,E|_S)$.
This follows from the Normal Form Theorem \cite[Thm.~3.9]{Green}. 
There is also a retraction homomorphism
\[
 r_S:G_\Gamma\to G_S
\]
which is obtained by substituting the trivial group for $G_v$ for all $v\in V-S$
\cite[Section 3]{Antolin}.

If $S\subseteq V$ is a subset (resp. a clique), then 
$G_S$ is called a \emph{special parabolic subgroup} (resp. a \emph{ special complete parabolic subgroup}). 
The conjugates in $G_\Gamma$ of the special (complete) parabolic subgroups are called 
(\emph{complete}) \emph{parabolic subgroups}.
We note that parabolic subgroups behave well. 
For $R,S\subseteq V$ and $a,b\in G_\Gamma$ we have
\[
\tag{1}
aG_{R}a^{-1}\subseteq bG_{S}b^{-1}\Rightarrow R\subseteq S
\]
see \cite[Corollary 3.8]{Antolin}. If $gG_Sg^{-1}\subseteq G_S$, then by \cite[Lemma 3.9]{Antolin}
\[
\tag{2}
gG_Sg^{-1}=G_S.
\]

Let $X$ be a subset of $G_{\Gamma}$. If the set of all parabolic subgroups containing $X$ has a minimal element,
then this minimal parabolic subgroup containing $X$ is unique by the remarks above. In this case,
it is called the \emph{parabolic closure} of $X$ and denoted by $\Pc(X)$.
The parabolic closure always exists if $\Gamma$ is finite or if $X$ is finite \cite[Proposition 3.10]{Antolin}.

Let $H\subseteq G_\Gamma$ be a subgroup. We denote by $\Nor_{G_\Gamma}(H)$ the normalizer of $H$ in $G_\Gamma$. 
For a parabolic subgroup of $G_\Gamma$ there is a good description of the normalizer.
\begin{lemma}\cite[Lemma 3.12 and Proposition 3.13]{Antolin}
\label{normalizer}
\begin{enumerate}
\item Let $H\subseteq G_\Gamma$ be a subgroup. Suppose that the parabolic closure of $H$ in $G_\Gamma$ exists. Then 
$\Nor_{G_\Gamma}(H)\subseteq \Nor_{G_\Gamma}{(\Pc(H))}$.
\item Let $G_S$ be a non-trivial special parabolic subgroup of $G_\Gamma$. Then 
$\Nor_{G_\Gamma}(G_S)=G_{S\cup\lk(S)}$. 
\end{enumerate}
\end{lemma}

\section{Actions on cube complexes}
A detailed description of \CAT\ cube complexes can be found in \cite{Haefliger} and in \cite{Sageev}. Let $\mathcal{C}$ denote the class of finite dimensional \CAT\ cube complexes and let $\mathcal{A}$ denote the subclass of $\mathcal{C}$ consisting of simplicial trees.
Inspired by Serre's fixed point property $\F\mathcal{A}$, Bass introduced the property $\F\mathcal{A'}$ in \cite{Bass}. 
A group $G$ has property $\F\mathcal{A'}$ if every simplicial action of $G$ on every member $T$ of $\mathcal{A}$ is locally elliptic, 
i.e. if each $g\in G$ fixes some point on the tree $T$. 
A generalization of property $\F\mathcal{A'}$ was defined in \cite{Leder}. A group $G$ has property $\F\mathcal{C'}$ if every
simplicial action of $G$ on every member $X$ of $\mathcal{C}$ is locally elliptic, i.e. if each $g\in G$ fixes some point on 
$X$. Bass proved in \cite{Bass} that every profinite group has property $\F\mathcal{A'}$. His result was generalized by 
Alperin to compact groups in \cite{Alperin} and to almost connected locally compact groups in \cite{Alperin2}.

The next result was proved by Caprace in \cite[Theorem 2.5]{Caprace}.
\begin{proposition}
\label{compact}
Let $X$ be a locally Euclidean \CAT\ cell complex with finitely many isometry types of cells, and $L$ be a compact group acting 
as an abstract group on $X$ by cellular isometries. Then every element of $L$ is elliptic. 
In particular, every compact group has property $\F\mathcal{C}'$.  
\end{proposition}

We recall that a group $G$ is called \emph{divisible} if $\left\{g^n\mid g\in G\right\}=G$ holds 
for all integers $n\geq1$. Another result which we will need in order to prove Proposition A is the following.
\begin{lemma}\cite[Theorem 2.5 Claim 7]{Caprace}
\label{divisible}
Every divisible group has property $\F\mathcal{C}'$. 
\end{lemma}

The following result is due to Sageev and follows from the proof of Theorem 5.1 in \cite{Sageev}, see also \cite[Theorem A]{Leder}.
\begin{proposition}
\label{globalfixedpoint}
Let $G$ be a finitely generated group acting by simplicial isometries on 
a finite dimensional \CAT\ cube complex. If the $G$-action is locally elliptic, then $G$ has a global fixed point. 
\end{proposition}

The last fact we need for the proof of Proposition A concerning global fixed points is the following 
easy consequence of the Bruhat-Tits Fixed Point Theorem \cite[Lemma 2.1]{Marquis}.
\begin{lemma}
\label{boundedproduct}
Suppose that a group $H$ acts isometrically on a complete \CAT\ space.
If $H=H_1H_2\cdots H_r$ is a product 
of finitely many subgroups $H_j$ each fixing some point in $X$, then $H$ has a global fixed point.
\end{lemma}

\subsection{Graph products, cube complexes and the building}
Associated to finite dimensional graph products are certain finite dimensional \CAT\ cube complexes. We
briefly describe the construction of these spaces. For a graph product $G_\Gamma$ we consider the poset
\[
P=\left\{gG_T\mid g\in G_\Gamma\text{ and } T \text{ is a clique}\right\},
\]
ordered by inclusion (we recall that we allow empty cliques).
The group $G_\Gamma$ acts by left multiplication on this poset and hence simplicially on the flag complex 
$X_\Gamma$ associated to this coset poset.
This flag complex has a canonical cubical structure. With respect to this structure $X_\Gamma$ is the Davis realization of 
a right-angled building, \cite[Theorem 5.1]{Davis}. By \cite[Theorem 11.1]{Davis} the Davis realization of every building is a complete \CAT\ space. 
Hence $X_\Gamma$ is a finite dimensional \CAT\ cube complex, and $G_\Gamma$ acts isometrically on $X_\Gamma$.
The \emph{chambers} of $X_\Gamma$ correspond to the cosets of the trivial subgroup, i.e. to the elements of $G_\Gamma$.
The $G_\Gamma$-stabilizer of a chamber (a maximal cube) is therefore trivial.
The \emph{vertices} of $X_\Gamma$ correspond to the cosets of the $G_S$, where $S\subseteq V$ is an inclusion-maximal clique.
The action of $G_\Gamma$ on $X_\Gamma$ preserves the canonical cubical structure. 

One nice property of this action is the following: if a subgroup $H\subseteq G_\Gamma$ has a global fixed point in $X_\Gamma$, 
then there exists a vertex in $X_\Gamma$ which is fixed by $H$. 
This follows from the fact that the action is type preserving. Furthermore, the stabilizer of a vertex 
$gG_T$ is equal to $gG_Tg^{-1}$.
\begin{lemma}
\label{globalFP}
Let $G_\Gamma$ be a finite dimensional graph product and let $H$ be a subgroup. 
If the action of $H$ on the building $X_\Gamma$ is locally elliptic, then $H$ has a global fixed point.
\end{lemma}
\begin{proof}
For each finite subset $X\subseteq H$, the finitely generated group $\langle X\rangle$ acts locally elliptically on $X_\Gamma$.
Thus $\langle X\rangle$ has by Proposition \ref{globalfixedpoint} a fixed vertex $gG_S$, for some $g\in G_\Gamma$ and 
some maximal clique $S$. It follows that the parabolic closure of $X$ is of the form 
$\Pc(X)=gG_{S_X}g^{-1}$, where $S_X$ is a clique depending uniquely on $X$.
Since there is an upper bound on the size of cliques in $\Gamma$,
there exists a finite set $Z\subseteq H$ such that $S_Z$ is maximal
among all cliques $S_X$, for $X\subseteq H$ finite.
We claim that $H\subseteq\Pc(Z)$.

Let $h\in H$ and put $X=Z\cup\{h\}$. If we put 
$\Pc(X)=aG_{S_X}a^{-1}$ and $\Pc(Z)=bG_{S_Z}b^{-1}$, then 
\[
 aG_{S_X}a^{-1}\supseteq bG_{S_Z}b^{-1}
\]
because $X\supseteq Z$. Then $S_X\supseteq S_Z$ holds by \ref{ParabolicSubgroups}(1).
From the maximality of $S_Z$ we conclude that $S_Z=S_X$. Then 
$aG_{S_X}a^{-1}=bG_{S_Z}b^{-1}$ by \ref{ParabolicSubgroups}(2).
It follows that $H\subseteq \Pc(Z)=bG_{S_Z}b^{-1}$,
and thus $H$ has a global fixed point.
\end{proof}

\section{The proofs of Proposition A and Theorem B}

\begin{NewPropositionA}
Let $\varphi:L\to G_\Gamma$ be an abstract homomorphism from an almost connected locally compact group $L$ to a finite dimensional graph product $G_\Gamma$. Then $\varphi(L)$ lies in a complete parabolic subgroup of $G_\Gamma$.
\end{NewPropositionA}
\begin{proof}
The group $L$ acts via
\[
L\to G_\Gamma\to\Isom(X_\Gamma)
\]
isometrically and simplicially on the right-angled building $X_\Gamma$.

Suppose first that $L$ is compact. Then the $L$-action is by Proposition \ref{compact} locally elliptic. 
Hence there is global fixed point by Lemma~\ref{globalFP}.

Suppose next that $L$ is connected. By Iwasawa's decomposition \cite[Theorem 13]{Iwasawa} we have
\[
L=H_1H_2\cdots H_r K,
\]
where $K$ is a connected compact group and $H_i\cong\mathbb{R}$ for $i=1, \ldots, r$. 
Each group $H_j$ has a fixed point by Lemma \ref{divisible},
and $K$ has a fixed point by the result in the previous paragraph.
Hence $L$ has a fixed point by Lemma~\ref{boundedproduct}.

Now we consider the general case.
If $L$ is almost connected, then the identity component $L^\circ$ has a global fixed point
 by the previous paragraph. The fixed point set $Z\subseteq X_\Gamma$ of $L^\circ$ is a convex
 \CAT\ cube complex, because the $L$-action is simplicial and type-preserving.
 By Proposition~\ref{compact}, the action of $L/L^\circ$ on $Z$ is locally
 elliptic. Hence the action of $L$ on $X$ is locally elliptic as well. Another application
 of Lemma ~\ref{globalFP} shows that $L$ has a global fixed point.
\end{proof}

Now we may prove Theorem B.

\begin{NewTheoremB}
Let $\varphi$ be an abstract homomorphism from a locally compact group $L$ to a finite dimensional graph product $G_\Gamma$. 
Then either $\varphi$ is continuous, or $\varphi(L)$ lies in a 
conjugate of a parabolic subgroup $G_{S\cup\lk(S)}$, where $S\neq\emptyset$ is a clique.
If every composite $L\xrightarrow{\varphi}G_\Gamma\xrightarrow{r_v} G_v$ is continuous, then $\varphi$ is 
continuous.
\end{NewTheoremB}
\begin{proof}
Let $L^\circ$ be the connected component of the identity in $L$. We distinguish two cases.

\smallskip\noindent\emph{Case 1: $\varphi(L^\circ)$ is not trivial.}

By Proposition A we know that $\varphi(L^\circ)\subseteq gG_Tg^{-1}$ where $T\subseteq V$ is a clique
and $g\in G_\Gamma$.
Hence ${\Pc}(\varphi(L^\circ))=hG_{S}h^{-1}$, where $\emptyset\neq S\subseteq T$ and $h\in G_\Gamma$.
Since $\varphi(L)$ normalizes $\varphi(L^\circ)$, we have 
by Lemma~\ref{normalizer} that $\varphi(L)\subseteq\Nor_{G_\Gamma}(\Pc(\varphi(L^\circ)))$.
This normalizer is of the form $h G_{S\cup\lk(S)} h^{-1}$, for some $h\in G_\Gamma$.
We note that in Case 1, the homomorphism $\varphi$ is not continuous, since the image of a connected group under continuous map is always connected and a connected subgroup of a discrete group is trivial. 

\smallskip\noindent\emph{Case 2: $\varphi(L^\circ)$ is trivial.}

Then $\varphi$ factors through an abstract homomorphism
$\psi:L/L^\circ\to G_\Gamma$, and $L/L^\circ$ is a totally disconnected locally compact group.
By van Dantzig's Theorem \cite[III\S 4, No. 6]{Bourbaki} there exists a compact open subgroup $K$ in $L/L^\circ$. 

\smallskip\noindent\emph{Subcase 2a: There is a compact open subgroup $K\subseteq L/L^\circ$ such that $\psi(K)$ is trivial.}
Then the kernel of $\psi$ is open in $L/L^\circ$ and hence $\psi$ and $\varphi$ are continuous. 

\smallskip\noindent\emph{Subcase 2b: There is no compact open subgroup $K\subseteq L/L^\circ$ such that $\psi(K)$ is trivial.}

Let $\mathcal K$ denote the collection of all compact open subgroups of $L/L^\circ$.
We note that $L$ acts on $\mathcal K$ by conjugation.
For $K\in\mathcal K$ we put $\Pc(\psi(K))=gG_{S_K}g^{-1}$.
Thus $S_K$ is a clique in $\Gamma$ which depends uniquely on $K$. 
We choose $M\in\mathcal K$ in such a way that $S_M$ is minimal and we note that $S_M\neq\emptyset$.
Given $a\in L/L^\circ$ we have 
$M\cap aMa^{-1}\in\mathcal K$ and
\[\Pc(\psi(aMa^{-1}))=\psi(a)\Pc(\psi(M))\psi(a)^{-1}.\]
From \ref{ParabolicSubgroups}(1) and
\[
\Pc(\psi(M))\supseteq\Pc(\psi(M)\cap\psi(aMa^{-1}))\subseteq\Pc(\psi(aMa^{-1}))
\]
we obtain that
\[
S_M\supseteq S_{M\cap aMa^{-1}}\subseteq S_{aMa^{-1}}.
\]
Since both $S_{aMa^{-1}}$ and $S_M$ are minimal we conclude that \[S_{M}=S_{M\cap aMa^{-1}}=S_{aMa^{-1}}\]
and that
\[\Pc(\psi(M))=\Pc(\psi(aMa^{-1}))=\psi(a)\Pc(\psi(M))\psi(a)^{-1}.\] 
Therefore $\psi(a)$ normalizes $\Pc(\psi(M))$, whence
\[
\varphi(L)= \psi(L/L^\circ)\subseteq hG_{S_M\cup\lk(S_M)}h^{-1},
\]
for some $h\in G_\Gamma$ by Lemma~\ref{normalizer}.

Suppose now towards a contradiction that $\varphi$ is not continuous, but that each composite $L\xrightarrow{\varphi}G_\Gamma\xrightarrow{r_v} G_v$ is continuous.
Then $\varphi(L)\subseteq gG_{S\cup\lk(S)}g^{-1}$ for some nonempty clique $S$.
There is a direct product decomposition
\[
 G_{S\cup\lk(S)}=G_S\times G_{\lk(S)}=\prod_{v\in S}G_v\times G_{\lk(S)}
\]
and therefore $\varphi$ factors as a product of commuting homomorphisms
\[
 \varphi(a)=g\prod_{v\in S}\varphi_v(a)\varphi_{\lk(S)}(a)g^{-1},
\]
with $\varphi_v=\varphi\circ r_v$ and $\varphi_{\lk(S)}=\varphi\circ r_{\lk(S)}$.
Here we use the retractions $r$ introduced in Section~\ref{GraphProductsSection}.
Since the $\varphi_v$ are all continuous, $\varphi_{\lk(S)}$ is not continuous.
Hence we find a clique $T\subseteq\lk(S)$ such that 
$\varphi_{\lk(S)}(L)\subseteq hG_{T\cup(\lk(T)\cap\lk(S))}h^{-1}$.
But then $S\cup T$ is a clique (because $T\subseteq\lk(S)$) which is strictly bigger than $S$.
If we continue in this fashion, we end up after finitely many steps with an empty link,
because $\Gamma$ has finite dimension. Thus $\varphi$ is a finite product of 
commuting continuous homomorphisms, and therefore itself continuous. This is a contradiction.
\end{proof}
The referee has pointed out that if every composite $L\xrightarrow{\varphi}G_\Gamma\xrightarrow{r_v} G_v$ is continuous, then $\varphi$ is 
continuous, whether or not the graph product if finite dimensional (see proof of Theorem 3.3 in \cite{Conner}).

\section{The proof of Proposition C}
We consider abstract homomorphisms from locally compact groups $L$ into groups $G$ which are in the class $\G$.
We recall from the introduction that for such a group $G$, every torsion subgroup $T\subseteq G$ is finite, and every
abelian subgroup $A\subseteq G$ is a (possibly infinite) direct sum of cyclic groups.
In particular, such a group $G$ has no nontrivial divisible abelian subgroups.
\begin{NewPropositionC}
Let $\varphi$ be an abstract homomorphism from a locally compact group $L$ to a group $G$ in the class $\mathcal{G}$. Then $\varphi$ factors through the canonical projection $\pi:L\to L/L^\circ$. If $L$ is almost connected, 
then $\varphi(L)$ is finite.  
\end{NewPropositionC}
\begin{proof}
We first show that every homomorphism $\rho:K\to G$ has finite image if $K$ is compact.
Suppose that $g\in K$. We claim that $\rho(g)$ has finite order.
The subgroup $H=\overline{\langle g\rangle}$ is compact abelian,
whence $\rho(H)=F\times \mathbb Z^{(J)}$, where $F$ is a 
finite abelian group. By Dudley's result
\cite{Dudley}, a compact group has no nontrivial free abelian quotients. Therefore 
$\rho(H)$ is finite and in particular, $\rho(g)$ has finite order. 
Since $G$ contains no infinite torsion groups, $\rho(K)$ is finite.

Now we show that $\varphi(L^\circ)$ is trivial.
By Iwasawa's Theorem \cite[Theorem 13]{Iwasawa} there is a decomposition $L^\circ=H_1\cdots H_r K$,
where $H_j\cong\mathbb R$ for $j=1,\ldots,r$ and where $K$ is a compact connected group.
The groups $H_1,\ldots,H_r$ are abelian and divisible. From our assumptions on the class $\G$
we see that the abelian groups $\varphi(H_j)$ are trivial, for $j=1,\ldots,r$.
The compact group $K$ is connected and therefore divisible \cite[Theorem 9.35]{HoMo}.
A finite divisible group is trivial, and therefore $\varphi(K)$ is trivial as well. 
This shows that $\varphi(L^\circ)$ is trivial.

The first paragraph of the present proof shows then that $\varphi(L)$
is finite if $L/L^\circ$ is compact.
\end{proof}

\section{The proof of Theorem D}
We are now ready to prove Theorem D.

\begin{NewTheoremD}
Let $\varphi$ be an abstract homomorphism from a locally compact group $L$ to a group $G$ in the class $\mathcal{G}$. 
Then either $\varphi$ is continuous, or $\varphi(L)$ lies in the normalizer of a finite non-trivial subgroup of $G$.
\end{NewTheoremD}
\begin{proof}
Let $L^\circ$ be the connected component of the identity in $L$. By Proposition C the homomorphism
$\varphi$ factors through a homomorphism $\psi:L/L^\circ\to G$.
The totally disconnected locally compact group $L/L^\circ$ contains by van Dantzig's Theorem 
\cite[III\S 4, No. 6]{Bourbaki} compact open subgroups. We distinguish two cases.

\smallskip\noindent\emph{Case 1: $\psi(K)$ is trivial for some compact open subgroup $K\subseteq L/L^\circ$.}

Then the kernel of $\psi$ is open and therefore $\psi$ and $\varphi$ are continuous.

\smallskip\noindent\emph{Case 2: $\psi(K)$ is nontrivial for every compact open subgroup $K\subseteq L/L^\circ$.}

By Proposition C, the image $\psi(K)$ of such a group $K$ is finite. 
Among the compact open subgroups of $L/L^\circ$ we choose $M$ such that $\psi(M)$ is minimal.
Given $g\in L/L^\circ$, we have then that $\psi(gMg^{-1})=\psi(M\cap gMg^{-1})=\psi(M)$. It follows that 
$\psi(g)$ normalizes $\psi(M)$.
\end{proof}

\section{Some remarks on the class $\G$}

In this last section we show that the class $\mathcal{G}$ contains many groups.

\begin{proposition}\label{LargeClass}
The class $\G$ is closed under passage to subgroups,
under passage to finite products, and under passage to arbitrary coproducts.
\end{proposition}
\begin{proof}
If $H\subseteq G\in\G$, then clearly $H\in\G$.
If $G_1,\ldots,G_r\in\G$ and if $T\subseteq \prod_{j=1}^rG_j$ is a torsion group,
then the projection $\pi_j(T)=T_j\subseteq G_j$ is also a torsion group.
Hence $T\subseteq\prod_{j=1}^rT_j$ is finite.
Similarly, if $A\subseteq \prod_{j=1}^rG_j$ is abelian, then 
$A$ is contained in the product $\prod_{j=1}^r\pi_j(A)$,
which is a direct sum of a finite abelian group and a free abelian group.
Hence $A$ itself is a direct sum of a finite abelian group and a free abelian group.
Finally suppose that $(G_j)_{j\in J}$ is a family of groups in $\G$.
By Kurosh's Subgroup Theorem \cite{Baer}, 
every subgroup of the coproduct $\coprod_{j\in  J}G_j$
is itself a coproduct $F*\coprod_{j\in J}g_jU_jg_j^{-1}$,
were $F$ is a free group, $U_j\subseteq G_j$ is a subgroup
and the $g_j$ are elements of $\coprod_{i\in  J}G_i$.
If such a group is abelian, then it is either cyclic or conjugate to a subgroup of
one of the free factors.
\end{proof}

\begin{proposition}
\label{Hyperbolic}
Every hyperbolic group $G$ is in the class $\mathcal{G}$.
\end{proposition}
\begin{proof}
By a theorem of Gromov \cite[Chap.8 Cor.36]{Gromov}, every torsion subgroup of a hyperbolic group is finite. 
Furthermore, every abelian subgroup of a hyperbolic group is finitely generated.
\end{proof}

\begin{proposition}
\label{Artin}
Let $A$ be an Artin group. If $A$ is a right-angled Artin group or an Artin group of finite type, 
then  $A$ is torsion free and every abelian subgroup of $A$ is finitely generated. 
\end{proposition}
\begin{proof}
Every right-angled Artin group is torsion free by \cite[Corollary 3.28]{Green}.
Moreover $A$ is a \CAT\ group, see \cite{Charney4}. Hence every abelian subgroup of $A$ is finitely generated, 
see \cite[II Corollary 7.6]{Haefliger}. 
If $A$ is an Artin group of finite type, then $A$ is torsion free by \cite{Brieskorn}.
By \cite[Corollary 4.2]{Bestvina}, every abelian subgroup of $A$ is finitely generated.
\end{proof}
We note that it is an open question if every Artin group is torsion free \cite[Conjecture 12]{Charney}.

\begin{proposition}
\label{Coxeter}
Let $W$ be a Coxeter group. Then every torsion subgroup of $W$ is 
finite and every abelian subgroup of $W$ is finitely generated. 
\end{proposition}
\begin{proof}
It was proved in \cite[Theorem 14.1]{Moussong} that Coxeter groups are \CAT\ groups. 
Hence every abelian subgroup of $W$ is finitely generated \cite[II Corollary 7.6]{Haefliger} and the order
of finite subgroups of $W$ is bounded \cite[II Corollary 2.8(b)]{Haefliger}. 
Let $T\subseteq W$ be a torsion group. Since $W$  is a linear group \cite[Corollary 6.12.11]{Davis2} 
and since every finitely generated linear torsion group is finite \cite[I]{Schur}, it follows that every
finitely generated subgroup of $T$ is finite. Since the order of finite subgroups of $T$ is bounded, $T$ is finite. 
\end{proof}

\begin{proposition}
\label{OutFn}
The groups $\GL_n(\mathbb{Z})$, the groups $\Out(F_n)$ of outer automorphisms of free groups and the 
mapping class groups $\Mod(S_g)$ of orientable surfaces of genus $g$ are in the class $\mathcal{G}$.
\end{proposition}
\begin{proof}
Since $\GL_n(\mathbb{Z})$ is a linear group, it follows that every finitely generated torsion subgroup  is 
finite \cite[I]{Schur}. Since the order of finite subgroups in $\GL_n(\mathbb{Z})$ is bounded \cite{Minkowski},
we obtain that every torsion subgroup of $\GL_n(\mathbb{Z})$ is finite. 
It was proved in \cite{Malcev} that every abelian subgroup of $\GL_n(\mathbb{Z})$ is finitely generated. 
Hence $\GL_n(\mathbb{Z})$ is in the class $\mathcal{G}$.

The kernel of the map $\Out(F_n)\to\GL_n(\mathbb{Z})$ which is induced by the abelianization of $F_n$ is torsion free \cite{Baumslag}. Since every torsion subgroup of $\GL_n(\mathbb{Z})$ is finite, it follows that every torsion subgroup of $\Out(F_n)$ is finite. Every abelian subgroup of $\Out(F_n)$ is finitely generated, see \cite{Lubotzky}. Thus $\Out(F_n)$ is in the class $\mathcal{G}$.

It was proved in \cite[Theorem A]{Birman} that every abelian subgroup of $\Mod(S_g)$ is finitely generated. Further, 
it was proved in \cite[Theorem 1]{Nikolaev} that $\Mod(S_g)$ is a linear group. Hence every finitely generated torsion subgroup is finite \cite[I]{Schur}. We know by \cite{Hurwitz} that the order of finite subgroups in $\Mod(S_g)$ is bounded. 
Therefore every torsion subgroup of $\Mod(S_g)$ is finite. Thus $\Mod(S_g)$ is in the class $\mathcal{G}$.
\end{proof}

\end{document}